\documentclass[reqno]{amsart}

\usepackage{amssymb}
\usepackage{graphicx}
\usepackage{enumerate}

\usepackage{slashed,enumerate}

\usepackage[utf8]{inputenc}

\title[Herman-Avila-Bochi formula for higher dimensional cocycles]
{A Herman-Avila-Bochi formula for higher dimensional pseudo-unitary and hermitian-symplectic Cocycles}

\author{Christian Sadel}
\address{University of British Columbia, 1984 Mathematics Road, Vancouver, BC, \mbox{V6T\,1Z2}, Canada}
 \email{csadel@math.ubc.ca}

\newtheorem{theorem}{Theorem}
\newtheorem{lemma}[theorem]{Lemma}
\newtheorem{coro}[theorem]{Corollary}
\newtheorem{prop}[theorem]{Proposition}
 
\newtheorem{defini}[theorem]{Definition}

\newtheorem*{rem}{Remark}

\newcommand{\C}{{\mathbf C}}

\newcommand{\Y}{{\mathbf Q}}

\newcommand{\Aa}{{\mathcal A}}
\newcommand{\Bb}{{\mathcal B}}
\newcommand{\Cc}{{\mathcal C}}
\newcommand{\Dd}{{\mathcal D}}

\newcommand{\Gg}{{\mathcal G}}

\newcommand{\Ii}{{\mathcal I}}
\newcommand{\Jj}{{\mathcal J}}

\newcommand{\Ll}{{\mathcal L}}

\newcommand{\Qq}{{\mathcal Q}}
\newcommand{\Rr}{{\mathcal R}}

\newcommand{\Tt}{{\mathcal T}}
\newcommand{\Uu}{{\mathcal U}}

\newcommand{\CC}{{\mathbb C}}
\newcommand{\DD}{{\mathbb D}}

\newcommand{\GG}{{\mathbb G}}

\newcommand{\NN}{{\mathbb N}}

\newcommand{\RR}{{\mathbb R}}
\newcommand{\Sb}{{\mathbb S}}

\newcommand{\VV}{{\mathbb V}}
\newcommand{\WW}{{\mathbb W}}
\newcommand{\XX}{{\mathbb X}}

\newcommand{\ZZ}{{\mathbb Z}}

\newcommand{\Aaa}{{\mathfrak A}}

\newcommand{\one}{{\bf 1}}
\newcommand{\nul}{{\bf 0}}

\newcommand{\qtx}[1]{\quad\text{#1}\quad}

\newcommand{\Mat}{{\rm Mat}}
\newcommand{\SL}{{\rm SL}}
\newcommand{\GL}{{\rm GL}}
\newcommand{\Sp}{{\rm Sp}}

\newcommand{\Her}{{\rm Her}}
\newcommand{\HSp}{{\rm HSp}}
\newcommand{\Un}{{\rm U}}

\newcommand{\pmat}[1]{\begin{pmatrix} #1  \end{pmatrix}}
\newcommand{\smat}[1]{\left( \begin{smallmatrix} #1  \end{smallmatrix} \right)}
\newcommand{\diag}{{\rm diag}}

\numberwithin{equation}{section}

\numberwithin{theorem}{section}

\begin{document}
\begin{abstract}
A Herman-Avila-Bochi type formula is obtained for the
average sum of the top $d$ Lyapunov exponents over a one-parameter family of
$\GG$-cocycles, where $\GG$ is the group that leaves a certain, non-degenerate hermitian form 
of signature $(c,d)$ invariant.
The generic example of such a group is the pseudo-unitary group $\Un(c,d)$ or in the case $c=d$, the hermitian-symplectic
group $\HSp(2d)$ which naturally appears for cocycles related to Schr\"odinger operators.
In the case $d=1$, the formula for $\HSp(2d)$ cocycles reduces 
to the Herman-Avila-Bochi formula for $\SL(2,\RR)$ cocycles.
\end{abstract}

\maketitle

\setcounter{tocdepth}{1}

\tableofcontents

\section{Introduction}

A fundamental problem in the Theory of dynamical systems is the explicit calculation of Lyapunov exponents.
The Herman-Avila-Bochi formula for $\SL(2\RR)$ cocycles is a remarkable formula giving an average of Lyapunov exponents
over a family of $\SL(2,\RR)$ cocycles.
It was first partly obtained as an inequality by Herman \cite{Her} and 
later proved to be an equality by Avila and Bochi \cite{AB}.
More recently, a different proof was given in \cite{BDD}.

I will show that there is a Herman-Avila-Bochi type formula 
for $\GG$-cocycles where $\GG$ is a matrix group that leaves a non-degenerate hermitian form $h(v,w)=v^* \Gg w$ invariant, i.e.
$\Gg$ is an invertible, hermitian matrix. The group $\SL(2,\RR)$ is a special case as it leaves the hermitian form 
given by $G=\smat{0&i\\-i & 0}$ invariant.
If $\Gg$ has signature $(c,d)$ \footnote{i.e. $\Gg$ has $c$ positive and $d$ negative eigenvalues} then one obtains a formula
for the average sum of the first $d$ Lyapunov exponents over a one-parameter family of cocycles.
By conjugation one only has to consider the pseudo-unitary group $\Un(c,d)$.
In fact, the proof in this group is very simple and a main step is a very well known Hilbert-Schmidt type decomposition for matrices within this group as shown in Theorem~\ref{th-LG1}.

Besides these groups an explicit form of the formula will also be stated for the hermitian-symplectic 
\footnote{This is different from the (complex)-symplectic group $\Sp(2d,\CC)$. 
$\Sp(2d,\CC)$ leaves the bilinear form given by $\Gg={G\otimes \one_d}$ invariant, i.e. $\Tt^\top \Gg \Tt = \Gg$, instead
of $\Tt^* \Gg \Tt=\Gg$.}
group $\HSp(2d)$ leaving the hermitian form given by \hbox{$G\otimes \one_d$} invariant.
 Such cocycles appear naturally in the theory of random and quasi-periodic Schr\"odinger operators.
In the case $d=1$, the formula for $\HSp(2d)$ cocycles reduces exactly to the Herman-Avila-Bochi formula for $\SL(2,\RR)$ cocycles.

Recent papers show some interest in higher dimensional cocycles and such quasi-periodic operators, e.g. 
\cite{AJS,DK1,DK2,HP,Sch}. 
The Herman-Avila-Bochi formula proved to be a useful tool for the theory of $\SL(2,\RR)$ cocycles and the presented
result may lead to some generalizations.

\subsection{Pseudo-unitary cocycles}

The Lorentz group or pseudo-unitary group of signature $(c,d)$, denoted by $\Un(c,d)$, is given by 
\begin{equation}
 \Un(c,d)=\left\{ \Tt\in\Mat(c+d,\CC)\,:\, \Tt^* \Gg_{c,d} \Tt = \Gg_{c,d}\right\}\;,\;
 \Gg_{c,d}=\pmat{\one_{c} & \nul \\ \nul & -\one_{d} }\,,
\end{equation}
i.e. it is the group of linear transformations, that leave the standard hermitian form of signature $(c,d)$ invariant.

\vspace{.2cm}

Let $(\XX,\Aaa,\mu)$ be a probability space and let $f:\XX\to\XX$ be a measure-preserving map,
i.e. for any $g\in L^1(d\mu)$ one has $\int_\XX g(f(x))d\mu(x)=\int_\XX g(x)d\mu(x)$.

Moreover, let $\Aa:\XX\to\Un(c,d)$ be a measurable map such that 
$x\mapsto\ln\|\Aa(x)\|$ is $\mu$-integrable, i.e. $\ln\|\Aa(\cdot) \|\in L^1(d\mu)$.
The set of such functions $\Aa$ will be denoted by $\Ll\Ii(\XX,\Un(c,d))$ (for logarithmic integrable).
This condition is sufficient for the existence of the Lyapunov exponents.
Then the pair $(f,\Aa)$ interpreted as map
\begin{equation}
 (f,\Aa)\,:\, (x,v) \in \XX \times\CC^{c+d} \mapsto (f(x),\Aa(x)v)\,.
\end{equation}
defines a $\Un(c,d)$-cocycle over $\XX$.
The iteration of this map gives $(f,\Aa)^n=(f^n,\Aa_n)$ where
\begin{equation}
 \Aa_n(x)=\Aa(f^{n-1}(x))  \cdots \Aa(f(x)) \Aa(x)\,.
\end{equation}
The Lyapunov exponents are defined by
\begin{equation}
 L_k(f,\Aa)= \lim_{n\to\infty} \frac1n \int_\XX \ln(\sigma_k(\Aa_n(x)))\,d\mu(x)
\end{equation}
where $\sigma_k(\Aa)$ denotes the $k$-th singular value of $\Aa$. We also define
\begin{equation}
 L^k(f,\Aa):=\sum_{j=1}^k L_j(f,\Aa)=L_1(f,\Lambda^k \Aa)\,.
\end{equation}
The existence of the Lyapunov exponents is guaranteed by Kingman's subadditive theorem, in fact,
for $\mu$ almost every $x\in\XX$ the function
\begin{equation}
 L^k(f,\Aa,x)=\lim_{n\to\infty} \frac1n\ln\|\Lambda^k(\Aa_n(x))\|\,
\end{equation}
exists and $ L^k(f,\Aa) =  \int_\XX L^k(f,\Aa,x)\,d\mu(x)$.

\vspace{.2cm}

I will consider averages of a one-parameter family by multiplying with certain unitary matrices.
Therefore, for $\theta\in[0,1]$ let us define
\begin{equation}
 \Uu^{(c,d)}_\theta =
 \pmat{e^{2\pi i \theta} \one_{c} & \nul \\ \nul & \one_{d}}\;\in\;\Un(c,d)\,.
\end{equation}

Analogue to the function $N(A)$ for $A\in \SL(2,\RR)$ introduced in \cite{AB}, for $\Tt\in\GL(m,\CC)$
and $\NN\ni r\leq m$ let us define the functions
\begin{equation}
 N_r(\Tt) = \sum_{i=1}^r \ln\left(\frac{\sigma_i(\Tt)+(\sigma_i(\Tt))^{-1}}2 \right)\,.
\end{equation}

Since $\sigma_i(\Tt)\geq 1$ for $i=1,\ldots,d$ for $\Tt\in\Un(c,d)$ (cf. Theorem~\ref{th-LG1}) one has
\begin{equation} \label{eq-rel-N}
 \ln(\|\Lambda^d \Tt\|)-d\ln(2) \leq N_d(\Tt) \leq \ln(\|\Lambda^d \Tt\|)\,.
\end{equation}

One obtains the following analogue to the Herman-Avila-Bochi formula.

\begin{theorem}\label{th-U-1}
 Let $\Aa\in \Ll\Ii(\XX,\Un(c,d))$, then one has
 \begin{equation}\label{eq-st-U-0}
  \int_0^1 L^{d}(f,\Uu^{(c,d)}_\theta\Aa)\,d\theta\,=\, \int_0^1 N_{d}(\Aa(x))\,d\mu(x)\;.
 \end{equation}
\end{theorem}

This will follow immediately from the definition of $L^d$, equation \eqref{eq-rel-N} and the 
following fact:

\begin{theorem}\label{th-U-2}
Let $\Tt_1,\ldots,\Tt_n \in \Un(c,d)$ then one finds
\begin{equation}\label{eq-st-U-1}
 \int_0^1 N_{d}(\Uu^{(c,d)}_\theta \Tt_1 \Uu^{(c,d)}_\theta \Tt_2\ldots \Uu^{(c,d)}_\theta \Tt_n)\,d\theta\,=\,
 \sum_{j=1}^n N_{d}(\Tt_j)
\end{equation}
and
\begin{equation} \label{eq-st-U-2}
 \int_0^1 \ln(\rho(\Lambda^{d}(\Uu^{(c,d)}_\theta \Tt_1 \Uu^{(c,d)}_\theta \Tt_2\ldots \Uu^{(c,d)}_\theta \Tt_n)))\,d\theta\,=\,
 \sum_{j=1}^n N_{d}(\Tt_j)
\end{equation}
where $\rho(\cdot)$ denotes the spectral radius.
\end{theorem}

\begin{rem} 
Let $\Tt\in \Un(c,d)$, $c\geq k > d$, then  $\sigma_k(\Tt)=1$ for $c\geq k > d$. 
Hence, one actually obtains $N_k(\Tt)=N_{d}(\Tt)$ and $L^k(f,\Aa)=L^{d}(f,\Aa)$ for any $k$ between $d$ and $c$.
Therefore, one could replace $d$ by any $k$ between $d$ and $c$ in 
\eqref{eq-st-U-0} and \eqref{eq-st-U-1}.
\end{rem}

For groups leaving a general non-degenerate hermitian form invariant one immediately obtains the following.

\begin{coro}
Let $\Gg$ be any invertible, $(c+d)\times(c+d)$ hermitian-matrix with signature $(c,d)$ and $\GG$ the group of matrices
leaving the form $v^* \Gg w$ invariant, i.e. $\GG=\left\{ \Tt \,:\,\Tt^* \Gg \Tt = \Gg \right\}$.
Then, there exists an invertible matrix $\Bb$ such that
for $\Aa\in\Ll\Ii(\XX,\GG)$ one has
\begin{equation}\label{eq-gen-GG}
 \int_0^1 L^d(f, \Bb^{-1} \Uu^{(c,d)}_\theta \Bb\, \Aa)\,d\theta\,=\,\int_\XX N_d(\Bb\,\Aa(x) \Bb^{-1})\,d\mu(x)\,.
\end{equation}
\end{coro}

\begin{proof}
By diagonalizing $\Gg$ and contractions we find an invertible matrix $\Bb$ such that 
$ \Gg = \Bb^*\Gg_{c,d} \Bb$ and thus
$\Bb \GG \Bb^{-1} = \Un(c,d)$. As $\Bb \Aa \Bb^ {-1} \in \Ll\Ii(\XX,\Un(c,d))$, \eqref{eq-gen-GG} follows.
\end{proof}

%%%%%%%%%%%%%%%%%%%%%%%%%%%%%%%%%%%%%%%%%%%%%%%%%%%%%%%%%%%%%%%%%%%%%%%%%%%%%%%%
%%%%%%%%%%%%%%%%%%%%%%%%%%%%%%%%%%%%%%%%%%%%%%%%%%%%%%%%%%%%%%%%%%%%%%%%%%%%%%%%

\subsection{Hermitian-symplectic cocycles}
 
The hermitian-symplectic group $\HSp(2d)$ is given by
\begin{equation}
 \HSp(2d)=\left\{ \Tt\in\Mat(2d,\CC)\,:\, \Tt^* \Jj \Tt = \Jj\right\}\qtx{where}
 \Jj=\pmat{\nul & \one_d \\ -\one_d & \nul}\,,
\end{equation}
hence it leaves the hermitian form given by $\Gg=i\Jj$ invariant. As $i\Jj$ has signature $(d,d)$,
$\HSp(2d)$ is conjugated to $\Un(d,d)$. The conjugation is actually unitary and given by the Cayley transform, i.e.
\begin{equation}
 \Cc\; \HSp(2d)\; \Cc^{*} = \Un(d,d) \qtx{where} \Cc=\frac1{\sqrt{2}} \pmat{\one_d & i \one_d  \\ \one_d & -i\one_d}\,\in\,\Un(2d)\,.
\end{equation}

Hermitian-symplectic cocycles appear naturally for Schr\"odinger operators on strips. More precisely,
assume $f$ to be invertible\footnote{If $f$ is not invertible one can still define operators on the 'half strip' $\ell^2(\NN)\otimes \CC^d$}, $T \in L^1(\XX,\GL(d,\CC),\,V \in L^1(\XX,\Her(d))$, where $\Her(d)$ denotes the set of hermitian $d\times d$ matrices. Then we get the following family of self-adjoint\footnote{the operators are self-adjoint
for $\mu$-almost all $x$ by the criterion in \cite{SB}} Schr\"odinger operators,
\begin{equation}
(H_x \Psi) = T(f^{n+1}(x)) \Psi_{n+1} + V(f^n(x)) \Psi_n + T^*(f^n(x))) \Psi_{n-1}\,
\end{equation}
on $\ell^2(\ZZ)\otimes\CC^d\,\ni\,\Psi=(\Psi_n)_n$,  $\Psi_n\in\CC^d$.
Solving $H_x \Psi=E\Psi$ leads to
\begin{equation}
 \pmat{T(f^{n+1}(x)) \Psi_{n+1} \\ \Psi_n} = \Tt^E(f^n(x)) \pmat{T(f^n(x)) \Psi_n \\ \Psi_{n-1}}
\end{equation}
where 
\begin{equation}
\Tt^E(x)=\pmat{(E\one_d-V(x))T^{-1}(x) & -T^*(x) \\ T^{-1}(x) & \nul}\,.
\end{equation}
It is not very hard to check that $\Tt^E(x)\in\HSp(2d)$.
The behavior of the generalized eigenvectors (not necessarily in $\ell^2$) is given by the cocycle
$(f,\Tt^E)$ and hence these cocycles are strongly related to the spectral theory of $H_x$. 

\vspace{.2cm}

The role of the unitaries $\Uu^{(c,d)}_\theta$ is played by the following rotation matrices: 
\begin{equation}
 \Rr_\theta= e^{-i \pi \theta} \Cc^* \Uu^{(d,d)}_\theta \Cc =
 \pmat{\cos(\pi\theta)\one_d & -\sin(\pi\theta) \one_d \\ \sin(\pi\theta) \one_d & \cos(\pi\theta) \one_d}
 \,\in\,\HSp(2d)\,.
\end{equation}
As $\Cc$ is unitary, $N_d(\Cc \Tt \Cc^*)=N_d(\Tt)$ and Theorems~\ref{th-U-1} and \ref{th-U-2} immediately imply the following.

\begin{theorem}\label{th-HSp-1}
 Let $\Aa\in \Ll\Ii(\XX,\HSp(2d))$, then
 \begin{equation} 
  \int_0^1 L^d(f,\Rr_\theta\Aa)\,d\theta\,=\,  \int_\XX N_d(\Aa(x))\,d\mu(x)\;.
 \end{equation}
 Let $\Tt_1,\ldots,\Tt_n\in\HSp(2d)$, then
 \begin{equation} \label{eq-st-1}
  \int_0^1 N_d(\Rr_\theta \Tt_1 \Rr_\theta \Tt_2 \ldots \Rr_\theta\Tt_n)\,d\theta
  =\sum_{j=1}^n N_d(\Tt_j)
 \end{equation}
and
\begin{equation} \label{eq-st-2}
  \int_0^1 \ln(\rho(\Lambda^d(\Rr_\theta \Tt_1 \Rr_\theta \Tt_2 \ldots \Rr_\theta\Tt_n)))\,d\theta
  =\sum_{j=1}^n N_d(\Tt_j)\,.
\end{equation}
\end{theorem}

\begin{rem}
 The case $d=1$ corresponds exactly to the Herman-Avila-Bochi formula for $\SL(2,\RR)$ cocycles as
 $\HSp(2)=\{e^{i\varphi} A\,:\,\varphi\in\RR, A\in\SL(2,\RR)\}=S^1\cdot \SL(2,\RR)$.
 %The possible phase factor $e^{i\varphi}$ has of course no influence on the Lyapunov exponents.
\end{rem}

%\begin{theorem}\label{th-U-3}
% Let $\alpha \not\in\QQ$. For a dense set $\Vv\subset C^\omega(\RR/\ZZ,\Un(d,d))$ and $\Aa\in\Vv$ one has
% $L_1(\alpha,\Aa)>0$.
%\end{theorem}

\section{Structure of pseudo-unitary matrices}

The most crucial fact is the following decomposition:

\begin{theorem}\label{th-LG1}
 For $\Tt\in \Un(c,d)$, $c\geq d$, there exist unitary $c\times c$ matrices $U_1,\, V_1 \in \Un(c)$ and
 unitary $d\times d$ matrices $U_2,\, V_2 \in \Un(d)$ and a real diagonal
 $d\times d$ matrix $\Gamma=\diag(\gamma_1,\gamma_2,\ldots,\gamma_d)>\nul$ with $\gamma_i\geq\gamma_{i+1}$ such that 
 \begin{equation}\label{eq-U(d,d)-str}
  \Tt=\pmat{U_1 & \nul \\ \nul & U_2} \pmat{\cosh(\Gamma) & \nul & \sinh(\Gamma) \\ 
  \nul & \one_{c-d} & \nul \\ \sinh(\Gamma) & \nul & \cosh(\Gamma)}
  \pmat{V_1 & \nul \\ \nul & V_2}\,.
 \end{equation}
 In particular, $\sigma_i(\Tt)=e^{\gamma_i}$, $\sigma_{d+c+1-i}(\Tt)=e^{-\gamma_i}=(\sigma_i(\Tt))^{-1}$, for $i=1,\ldots,d$, 
 and $\sigma_i(\Tt)=1$, for $d<i\leq c$, are the $d+c$ singular values of $\Tt$. Thus, the matrix $\Gamma$ is uniquely determined.
 If $\Tt=\pmat{A&B\\C&D}$, where $A \in \Mat(c,\CC),\,D\in\Mat(d,\CC)$, then one finds $D=U_2 \cosh(\Gamma) V_2$ implying
  \begin{equation}\label{eq-sigma-rel}
  \sigma_i(D)= \cosh(\gamma_i)=\frac12 \left(\sigma_i(\Tt) + (\sigma_{i}(\Tt))^{-1} \right)\,\quad
  \text{for $i=1,\ldots,d$}
 \end{equation}
 and
 \begin{equation}\label{eq-det-D}
  |\det(D)| = \det(\cosh(\Gamma))=\exp(N_d(\Tt))\,.
 \end{equation}
\end{theorem}

\begin{proof}
 Let $\Tt=\smat{A&B\\C&D}$ written in blocks of size $c$ and $d$. 
 Then $\Tt\in\Un(c,d)$ implies 
 \begin{equation}\label{eq-U(d,d)-id}
  B^*A=D^*C,\quad
 D^*D-B^*B=\one_d,\quad AA^*-BB^*=\one_c
 \end{equation}
For the last equation, note that $\Tt^*\in\Un(c,d)$ as well, hence $\Tt\Gg_{c,d}\Tt^* =\Gg_{c,d}$.\footnote{Indeed, as
$\Gg_{c,d} \Tt^* \Gg_{c,d} \Tt=\Gg_{c,d}^2=\one$, one finds 
$\Tt^{-1}=\Gg_{c,d}\Tt^* \Gg_{c,d}$ and $\Tt\Gg_{c,d}\Tt^*=\Gg_{c,d}^{-1}=\Gg_{c,d}$, giving $\Tt^*\in\Un(c,d)$.}
 The Hilbert Schmidt or singular value decomposition of $B \in \Mat(c\times d,\CC)$ gives 
 \begin{equation}\label{eq-B}
  B=U_1 \pmat{\sinh(\Gamma) \\ \nul} V_2
 \end{equation}
for some unitaries $U_1\in\Un(c),\,V_2\in\Un(d)$ and a diagonal $d\times d$ matrix $\Gamma$ as described above.
By the second equation in \eqref{eq-U(d,d)-id} it follows that $D^*D=V_2^* (\one+\sinh^2(\Gamma)) V_2=V_2^*\cosh^2(\Gamma)V_2$.
Hence, defining $U_2$ by
\begin{equation}\label{eq-D}
 D=U_2 \cosh(\Gamma) V_2
\end{equation}
one sees that $U_2\in\Un(d)$.
Similarly, defining $V_1\in\Mat(c,\CC)$ by
\begin{equation}\label{eq-A}
 A=U_1 \pmat{ \cosh(\Gamma) & \nul \\ \nul & \one_{c-d} } V_1
\end{equation}
and using the third equation in \eqref{eq-U(d,d)-id} one also sees that $V_1 V_1^*=\one_c$, thus
$V_1\in\Un(c)$.
Finally, using the first equation in \eqref{eq-U(d,d)-id} one obtains
\begin{equation}\label{eq-C}
 C={D^*}^{-1} B^* A = U_2 \pmat{\sinh(\Gamma) & \nul}  V_1\,.
\end{equation}
By \eqref{eq-B}, \eqref{eq-D}, \eqref{eq-A} and \eqref{eq-C}, equation \eqref{eq-U(d,d)-str} follows.
\end{proof}

%\vspace{.2cm}

Note, for the special case $c=d$, $\Tt\in\Un(d,d)$, this theorem yields
$$
\Tt=\pmat{U_1 & \nul \\ \nul & U_2}\pmat{\cosh(\Gamma) & \sinh(\Gamma) \\ \sinh(\Gamma) & \cosh(\Gamma) }
\pmat{V_1 & \nul \\ \nul & V_2}\,.
$$

Next we want to consider the actions on $d$-dimensional subspaces, therefore let
$G(d,c+d)$ be the Grassmannian of $d$-dimensional subspaces of $\CC^{c+d}$. Such a subspace $\VV\in G(d,c+d)$ is represented by
a $(c+d)\times d$ matrix $\Phi$ of full rank $d$, where the $d$ column vectors of $\Phi$ span $\VV$.
Two such matrices are equivalent, $\Phi_1\sim\Phi_2$, in the sense that they span the same subspace, if
$\Phi_1=\Phi_2 S $ for $S\in\GL(d,\CC)$. We denote an equivalence class by $[\Phi]_\sim\in G(d,c+d)$.
In fact, $G(d,c+d)$ can be considered as a quotient of Lie groups and defines a holomorphic manifold (cf. \cite[Appendix~A]{AJS}).
The group $\GL(c+d,\CC)$ and hence in particular the group $\Un(c,d)$ acts on $G(d,c+d)$ by
$\Tt [\Phi]_\sim:=[\Tt\Phi]_\sim$.

\begin{defini}
Let us define the following subset of $G(d,c+d)$,
\begin{equation}
 \Sb=\left\{ \left[ \pmat{M \\ \one_d} \right]_\sim\,:\,M\in\Mat(c\times d,\CC)\,,\,M^*M < \one_d \right\} \subset G(d,c+d)\,.
\end{equation} 
This set is the image of the classical domain
\begin{equation}
 R_I(c,d)=\{M\in\Mat(c\times d,\CC)\,:\,M^*M<\one_d\}
\end{equation}
under the holomorphic injection
\begin{equation}
 \varphi: \Mat(c\times d,\CC) \to G(d,c+d)\,,\quad \varphi(M)=\left[\pmat{M \\ \one_d} \right]\,.
\end{equation}
This map can be viewed as a holomorphic chart for $G(d,c+d)$.
\end{defini}

\begin{prop}\label{prop-U(d,d)-Sb}
 The action of $\Un(c,d)$ leaves $\Sb$ invariant, i.e. $\Tt[\Phi]_\sim \in\Sb$
 for $\Tt\in\Un(c,d), \,[\Phi]_\sim \in\Sb$.
\end{prop}
\begin{proof}
Let $M\in R_I(c,d)$, $\Tt\in\Un(c,d)$ and let
\begin{equation}
 \pmat{X\\Y} = \Tt \pmat{M \\ \one_d}\,\quad\,
 X\in\Mat(c\times d,\CC),\;Y\in\Mat(d,\CC)
\end{equation}
Then
\begin{equation}
 \pmat{X\\Y}^* \Gg_{c,d} \pmat{X\\Y} = 
 \pmat{M \\ \one_d}^* \Tt^* \Gg_{c,d} \Tt \pmat{M \\ \one_d} = 
 \pmat{M \\ \one_d}^* \Gg_{c,d} \pmat{M \\ \one_d}
\end{equation}
Hence, $X^*X-Y^*Y=M^*M-\one_d$ and $Y^*Y=X^*X+\one_d-M^*M>X^*X$ is invertible and
\begin{equation}
 \pmat{X\\Y} \sim \pmat{XY^{-1}\\\one_d}\,,\quad
 (XY^{-1})^* (XY^{-1})=\one_d - {Y^{-1}}^*(\one_d-M^*M) Y^{-1} < \one_d
\end{equation}
which finishes the proof.
\end{proof}

As a final note, let us remark the following.

\begin{rem} \label{rm-moebius}
The action of $\Un(c,d)$ on $\Sb$ corresponds to the M\"obius action on the classical domain $R_I(c,d)$.
Hence, for $M\in \Mat(c\times d,\CC)$ with $M^*M<\one_d$ and
$\Tt=\smat{A & B \\ C & D}$ define
\begin{equation}
 \Tt \cdot M = (AM+B)(CM+D)^{-1}\,,
\end{equation}
then
\begin{equation}
\Tt \varphi(M) = \varphi(\Tt \cdot M) \,.
\end{equation}
By the calculation in the proof of Proposition~\ref{prop-U(d,d)-Sb}, the inverse in the M\"obius action exists for 
$M\in R_I(c,d)$.
In fact, the group $\Un(c,d)$ represents exactly the biholomorphic maps on $R_I(c,d)$.
\end{rem}

\section{Analytic dependence  of invariant subspaces}

In this section we will finally prove the main theorems. They will follow from the mean value property
of harmonic functions. We may assume, without loss of generality, that $c\geq d$ as $\Un(c,d)$ and $\Un(d,c)$ are conjugated
groups to each other, $N_d(\Tt)=N_c(\Tt)$ for $\Tt\in\Un(c,d)$, 
and the conjugation maps $\Uu^{(c,d)}_\theta$ to $e^{2\pi i\theta} \Uu^{(d,c)}_{-\theta}$.

\vspace{.2cm}

For $z\in \CC$ let us define
\begin{equation}
 \Bb(z)=\pmat{z \one_c & \nul \\ \nul & \one_d}\,.
\end{equation}
We will consider the cocycles $\Bb(z)\Aa$ for $|z|<1$ and denote the unit disk by
 $\DD:=\{z\in\CC\,:\,|z|<1\}$.
It is obvious that for $z\in\DD$, $\Bb(z)$ maps $\overline{\Sb}$ into $\Sb$, i.e.
$\Bb(z) \overline\Sb \subset \Sb$.
Now let $\Tt_1,\ldots,\Tt_n \in\Un(c,d)$ and let
\begin{equation}
\Dd(z)=\Bb(z) \Tt_1 \Bb(z) \Tt_2 \ldots \Bb(z) \Tt_n \,.
\end{equation}
\begin{lemma}
There is a holomorphic function $\WW:\DD \to G(d,c+d)$ such that $\WW(z)$ is invariant under $\Dd(z)$ and $\WW(z)$ is
associated to the $d$ largest absolute values of eigenvalues of $\Dd(z)$.\footnote{$\WW(z)$ is basically a direct sum of generalized eigenspaces of $\Dd(z)$}
In particular, let $D_\WW(z)$ be the restriction of
$\Dd(z)$ on $\WW(z)$, then
\begin{equation}\label{eq-rh-d}
z \mapsto \ln(\rho(\Lambda^d \Dd(z)))=\ln |\det D_\WW(z) | \,
\end{equation}
is harmonic on $\DD$ and continuous on $\overline \DD$.
\end{lemma}

\begin{proof}
 As $\Dd(z) \overline \Sb\subset \Sb$, using the chart $\varphi$, we see that the image of the classical domain
 $R_I(c,d)$ under the M\"obius action of $\Dd(z)$ has compact closure in $R_I(c,d)$. 
 By the Earle-Hamilton fixed point theorem, the map $M\mapsto \Dd(z)\cdot M$ has a unique fixed point $M(z) \in R_I(c,d)$
 and $\Dd^n(z)\cdot M$ converges to $M(z)$ for any $M\in R_I(c,d)$.\footnote{in fact, the Carath\'eodory metric is contracted}
 Therefore, the subspace $\WW(z)=\varphi(M(z))=\lim_{n\to\infty} \varphi(\Dd^n(z) \cdot\nul)$ is invariant under $\Dd(z)$. 
 Clearly, for any $n$, $z\mapsto \Dd^n(z) \cdot \nul $ is holomorphic in $z\in\DD$. 
 As the family of these functions take only values in $R_I(c,d)$, it is a normal family by Montel's theorem. Thus, the limit is holomorphic as well. As the action of $\Dd^n(z)$ on $G(d,c+d)$ contracts a neighborhood of $\WW(z)$, 
 it is clear that $\WW(z)$ 
 is spanned by generalized eigenvectors of $\Dd(z)$ that correspond to the eigenvalues with largest absolute values.
 Choosing the column vectors of $\smat{M(z) \\ \one}$ as a basis for $\WW(z)$,
 $D_\WW(z)$ is represented as an invertible, holomorphic $d\times d$ matrix valued 
 function and one finds
 that $\ln(\rho(\Lambda^d \Dd(z))) = \ln(\rho(\Lambda^d D_\WW(z))) = \ln\left| \det D_\WW(z) \right|$
 is harmonic in $z\in\DD$.
 For the first equality, note that the spectral radius of $\Lambda^d \Dd$ is the product of the $d$ largest absolute values 
 of eigenvalues of $\Dd$ (counted with algebraic multiplicity).
\end{proof}

\begin{rem}
 Assume that $f$ is an invertible transformation, $\Aa\in\Ll\Ii(\XX;\Un(c,d))$. 
 The proof for the holomorphic dependence of $\WW(z)$ and the related harmonic dependence of
 $\ln(\rho(\Lambda^d \Dd(z)))$ can be modified to a proof of harmonic dependence of
 $L^d(f,\Bb(z)\Aa)$ on $z\in\DD$. In fact, $\XX\times\Sb$ is a $d$-conefield,
 $\Aa(x) \overline \Sb \subset \Sb$ shows that the cocycle $(f,\Bb(z)\Aa)$ is $d$-dominated and
 the unstable directions $\WW(x,z)=\lim_{n\to\infty} (\Bb(z)\Aa)_n(f^{-n}(x))[\smat{\nul\\\one}]_\sim$ 
 of the corresponding dominated splitting depend holomorphically on $z$, cf. \cite[Sections~3 and 6]{AJS}.
 Choosing the basis $\smat{\varphi^{-1} (\WW(x,z)) \\ \one }$,  the restriction  $\Bb(z)\Aa(x): \WW(x,z) \to \WW(f(x),z)$ can be written as $d\times d$ matrix $D_\WW(x,z)$, holomorphic in $z$ 
 and $\int_\XX \ln\|D_\WW(x,z)\|\,d\mu(x)<\infty$ uniformly in $z$. 
 Hence, $L^d(f,\Bb(z)\Aa)=\int_\XX \ln |\det D_\WW(x,z)|\,d\mu(x)$ is harmonic in $z$; cf. \cite[Section~2]{Av}.
\end{rem}

Now Theorem~\ref{th-U-2} follows easily:

\renewcommand{\proofname}{Proof of Theorem~\ref{th-U-2}}
\begin{proof}
As $\Bb(e^{2\pi i\theta})  = \Uu^{(c,d)}_\theta$, one has
\begin{equation}\label{eq-pr-U-0}
 \Dd(e^{2\pi i\theta})= \Uu^{(c,d)}_\theta\Tt_1 \Uu^{(c,d)}_\theta\Tt_2 \cdots \Uu^{(c,d)}_\theta\Tt_n
\end{equation}
By harmonicity of $\ln(\rho(\Lambda^d \Dd(z)))$ we have
\begin{equation}\label{eq-pr-U-1}
 \int_0^1 \ln(\rho(\Lambda^d \Dd(e^{2\pi i\theta})))\,d\theta\,=\,
 \ln(\rho(\Lambda^d\Dd(0)))\,.
\end{equation}
Now, using blocks of sizes $c$ and $d$, let
\begin{equation}
 \Tt_j = \pmat{A_j & B_j \\ C_j & D_j} \qtx{then}
 \Bb(0) \Tt_j = \pmat{\nul & \nul \\ C_j & D_j}
\end{equation}
and hence
\begin{equation}
 \Dd(0)=\Qq^{-1} \pmat{\nul & \nul \\ \nul & D_1\cdots D_n} \Qq \qtx{for}
 \Qq=\pmat{\one_c & \nul \\ D_n^{-1} C_n & \one_d}.
\end{equation}
In particular, using \eqref{eq-det-D} this gives
\begin{equation}\label{eq-pr-U-2}
 \ln(\rho(\Lambda^d \Dd(0))) = \ln|\det(D_1\cdots D_n)| = \sum_{j=1}^n \ln|\det(D_j)|= \sum_{j=1}^n N_d(\Tt_j)
\end{equation}
Now, \eqref{eq-pr-U-0}, \eqref{eq-pr-U-1} and \eqref{eq-pr-U-2} together give \eqref{eq-st-U-2}.
Using \eqref{eq-pr-U-0} and \eqref{eq-st-U-2} again one also obtains
\begin{align}
\int_0^1 N_d(\Dd(e^{2\pi i \theta}))\,d\theta & =
\int_0^1 \int_0^1 \ln \rho(\Lambda^d \Uu^{(c,d)}_{\theta'} \Dd(e^{2\pi i\theta}))\, d\theta\, d\theta'\,\notag \\
 &= \int_0^1 N_d(\Uu^{(c,d)}_{\theta'}\Tt_1)\,+\,\sum_{j=2}^n N_d(\Tt_j)\,d\theta' \,=\,
 \sum_{j=1}^n N_d(\Tt_j)
\end{align}
which shows \eqref{eq-st-U-1}\,.
\end{proof}

\renewcommand{\proofname}{Proof of Theorem~\ref{th-U-1}}

\begin{proof}
 By subadditivity we find for $\Aa\in\Ll\Ii(\XX,\Un(c,d))$
\begin{align}
& 0\leq \frac1n  \int_\XX N_d ((\Uu^{(c,d)}_\theta \Aa)_n(x))\,d\mu(x)  \leq
 \frac1n \int_\XX \ln \|\Lambda^d (\Uu^{(c,d)}_\theta \Aa)_n(x) \|\,d\mu(x) \notag \\
 & \quad \leq \int_\XX \ln\|\Lambda^d \Uu^{(c,d)}_\theta \Aa(x)\|\,d\mu(x)\, \leq d \int_\XX \ln \|\Aa(x)\|\,d\mu(x) < \infty
\end{align}
uniformly in $\theta$ and $n$.
Hence, using \eqref{eq-rel-N} we find by Dominated Convergence
 \begin{align}
&  \int_0^1 L^d(f,\Uu^{(c,d)}_\theta\Aa)\,d\theta = \int_0^1 \lim_{n\to\infty} \int_\XX \frac1n N_d((\Uu^{(c,d)}_\theta \Aa)_n(x))\,d\mu(x)\;d\theta
  \notag \\
& \quad = \lim_{n\to\infty} \int_\XX \int_0^1 \frac1n N_d(\Uu^{(c,d)}_\theta \Aa(f^{n-1} x) \cdots \Uu^{(c,d)}_\theta \Aa(x) )\,d\theta\,d\mu(x) \notag \\
& \quad = \lim_{n\to\infty} \int_\XX \frac1n \sum_{j=0}^{n-1} N_d(\Aa(f^j(x)))\,d\mu(x) = \int_\XX N_d(\Aa(x))\,d\mu(x)
 \end{align}
\end{proof}


\begin{thebibliography}{AAA}

%\bibitem[Ab]{Ab}
%M. Abate, 
%Iteration Theory of Holomorphic Maps on Taut Manifolds, 
%Mediterranean Press, Rende, Cosenza 1989

\bibitem[AB]{AB} A. Avila and J. Bochi,
{\sl A formula with some applications to the theory of Lyapunov exponents},
Israel J. Math. {\bf 131} (2002), 125-137

\bibitem[Av]{Av} A. Avila,
{\sl Density of positive Lyapunov exponents for $\SL(2,\RR)$-cocycles},
J. Amer. Math. Soc. {\bf 24} (2011), 999-1014


\bibitem[AJS]{AJS} A. Avila, S. Jitomirskaya and C. Sadel,
{\sl Complex one-frequency cocycles},
preprint, arXiv:1306.1605 (2013)

\bibitem[BDD]{BDD} 
A. Baraviera, J. Dias, P. Duarte,
{\sl On the Herman-Avila-Bochi formula for Lyapunov exponents of $\SL(2,\RR)$ cocycles.},
Nonlinearity {\bf 24} (2011), 2465

\bibitem[DK1]{DK1} P. Duarte, S. Klein,
{\sl Positive Lyapunov exponents for higher dimensional quasiperiodic cocycles},
preprint, arXiv:1211.4002 (2012)

\bibitem[DK2]{DK2} P. Duarte, S. Klein,
{\sl Continuity of the Lyapunov exponents for quasiperiodic cocycles},
preprint, arXiv:1305.7504 (2013)

\bibitem[Her]{Her} M. Herman,
{\sl Une m\'ethode pour minorer les exposants de Lyapounov et quelques exemples montrant
le caract\'ere local d'un th\'eor\'eme d' Arnold et de Moser sur le tore de dimension 2}, 
Comment. Math. Helvetici
{\bf 58} (1983), 453-502

\bibitem[HP]{HP} A. Haro, J. Puig,
{\sl A Thouless formula and Aubry duality for long-range
Schr\"odinger skew-products},
Nonlinearity {\bf 26} (2013), 1163–1187

\bibitem[Sch]{Sch} W. Schlag,
{\sl Regularity and convergence rates for the Lyapunov exponents of linear co-cycles,}
preprint, arXiv:1211.0648 (2012)

\bibitem[SB]{SB} H. Schulz-Baldes,
{\sl Geometry of Weyl theory for Jacobi matrices with matrix entries}
J. d' Analyse Math. {\bf 110} (2010), 129-165

\end{thebibliography}
\end{document}